\documentclass[a4paper,9pt]{amsart}
\usepackage{graphicx,amssymb,amstext,amsmath,latexsym,dsfont,comment,color}
\newtheorem{theorem}{Theorem}
\newtheorem{lemma}[theorem]{Lemma}

\newtheorem{definition}[theorem]{Definition}
\newtheorem{proposition}[theorem]{Proposition}

\DeclareMathOperator{\diam}{diam}
\DeclareMathOperator{\Hdim}{\mathcal{H}}

\newcommand{\R}{\mathbb{R}}
\newcommand{\N}{\mathbb{N}}

\begin{document}

\title{Uniform bounds for the heat content of open sets in Euclidean space}
\date{29 July 2016}

\author{M. van den Berg\ *,\ K. Gittins}
\thanks{*\ Partially supported by The Leverhulme Trust,
International Network Grant \emph{Laplacians, Random Walks, Bose
Gas, Quantum Spin Systems}}
\address{School of Mathematics\\ University of Bristol\\
University Walk\\ Bristol\\ BS8 1TW, U.K.}
\begin{abstract}
We obtain (i) lower and upper bounds for the heat content of an
open set in $\R^m$ with $R$-smooth boundary and finite Lebesgue
measure, (ii) a necessary and sufficient geometric condition for
finiteness of the heat content in $\R^m$, and corresponding lower
and upper bounds, (iii) lower and upper bounds for the heat loss
of an open set in $\R^m$ with finite Lebesgue measure.
\end{abstract}
\maketitle

\section{Introduction}
In this paper we obtain results for the heat content of an open
and bounded set $D$ in Euclidean space $\R^m,\ m=2,3,\cdots$,
where $D$ has initial temperature $1$ and the complement of $D$
has initial temperature $0$. We denote the fundamental solution of
the heat equation on $\R^m$ by
\begin{equation*}
p(x,y;t)=(4\pi t)^{-m/2}e^{-\vert x - y \vert^{2}/(4t)},
\end{equation*}
and define
\begin{equation}
u_{D}(x;t)=\int_{D} dy \, p(x,y;t). \label{e5}
\end{equation}
It is standard to check (see Chapter 2 in \cite{LCE}) that
\begin{equation}\label{e1}
\Delta u_{D} = \frac{\partial u_{D}}{\partial t}, \quad x \in
\mathbb{R}^{m}, \, t>0,
\end{equation}
and that
\begin{equation}\label{e2}
\lim_{t \downarrow 0}u_D(x;t)= \mathds{1}_{D}(x), \quad x \in
\mathbb{R}^{m}-\partial D,
\end{equation}
where $\partial D$ is the boundary of $D$.

We define the \emph{heat content of $D$ in $\mathbb{R}^{m}$ at
$t$} by
\begin{equation}\label{e3}
H_{D}(t)=\int_{D} dx \, u_{D}(x;t).
\end{equation}
Thus, if $u_{D}(x;t)$ represents the temperature at point $x \in
\mathbb{R}^{m}$ at time $t$ with initial condition \eqref{e2},
then the heat content of $D$ in $\mathbb{R}^{m}$ at time $t$
represents the amount of heat in $D$ at time $t$. By \eqref{e5}
and \eqref{e3} we see that
\begin{equation}\label{e6}
H_{D}(t) = \int_{D} dx \int_{D} dy \, p(x,y;t).
\end{equation}
By the heat semigroup property we have that
\begin{equation}\label{e7}
p(x,y;t)=\int_{\R^m}dz \, p(x,z;t/2)p(z,y;t/2).
\end{equation}
By Tonelli's Theorem and
\eqref{e5},\ \eqref{e6} and \eqref{e7} we
conclude that
\begin{equation*}
H_D(t)=||u_D(\cdot;t/2)||_{L^2(\R^m)}^2.
\end{equation*}
Preunkert \cite{mP03} defines the $L^2$-curve of the set $D$ as
the map $t\mapsto ||u_D(\cdot;t/2)||_{L^2(\R^m)}$. The results for
the $L^2$-curve in Theorem 2.4 of \cite{MPPP07} imply that if $D$
is an open, bounded subset of $\mathbb{R}^{m}$ with
$C^{1,1}$-boundary $\partial D$ then
\begin{equation}\label{e10}
H_{D}(t)=\vert D \vert - \pi^{-1/2}\mathcal{P}(D)t^{1/2} +
o(t^{1/2}),\ t\downarrow 0,
\end{equation}
where $|D|$ denotes the Lebesgue measure of $D$ and
$\mathcal{P}(D)$ denotes the perimeter of $D$. Note that since
$\partial D$ is Lipschitz,
$\mathcal{P}(D)=\mathcal{H}^{m-1}(\partial D)$, the
$(m-1)$-dimensional Hausdorff measure of the boundary (see Remark
(ii) on p.183 in \cite{EG92}).

Initial value problems of the type \eqref{e1}-\eqref{e2} have been
studied in the much wider context of operators of Laplace type on
compact Riemannian manifolds \cite{vdBG}. The results of that
paper imply that if $D$ is open, bounded in $\R^m$ with
$C^{\infty}$ boundary then there exist geometric invariant
$h_0,h_1,\cdots$ such that for any $J\in \N$,

\begin{equation}\label{e10a}
H_D(t)=\sum_{j=0}^Jh_jt^{j/2}+O(t^{(J+1)/2}),\   \, t\downarrow 0.
\end{equation}
Furthermore if $\partial D$ is oriented by a unit inward pointing
normal vector field and if $\{k_1(s),\cdots,k_{m-1}(s)\}$ are the
principal curvatures at $s \in \partial D$ then $h_0=|D|$,
$h_1=-\pi^{-1/2}\int_{\partial D}d\mathcal{H}^{m-1}(s)$, $h_2=0$
and
\begin{equation*}
h_3= \frac{8}{3\sqrt{\pi}}\int_{\partial
D}d\mathcal{H}^{m-1}(s)\left(\frac{5}{32}\left(\sum_{i=1}^{m-1}k_i(s)\right)^2+\frac{1}{16}\sum_{i=1}^{m-1}k_i^2(s)\right).
\end{equation*}
For further results in the Riemannian manifold setting we refer to
\cite{vdBGK} and \cite{vdBG1}.

We note that $p(x,y;t)$ is also the transition density of Brownian
motion $(B(t), t\ge 0, \mathbb{P}_x,x\in \R^m)$ associated to the
Laplacian. In fact $u_D(x;t)=\mathbb{P}_x(B(t)\in D)$, and
$|D|^{-1}H_D(t)$ is precisely this probability averaged over all
starting points with uniform density.

The knowledge of the asymptotic behaviour of $H_D(t),\ t\downarrow
0$ as in \eqref{e10} or \eqref{e10a} does not give any information
about its actual numerical value. Furthermore, due to the
discontinuity of $\mathds{1}_{D}$, the implementation of numerical
schemes is non-trivial in the small $t$ regime.

In this paper, we address this issue and obtain uniform bounds for
$H_D(t)$ under the hypotheses that either $D$ has $R$-smooth
boundary $\partial D$ and finite Lebesgue measure $\vert D \vert$,
or $D$ is an arbitrary open set satisfying a geometrical
integrability condition. These bounds are uniform in both the
geometrical data of $D$ and $t$, and the $m$-dependent constants
are explicit.
\begin{definition}
An open set $D \subset \mathbb{R}^{m}$, $m\geq 2$, has $R$-smooth
boundary if at any point $x_{0} \in \partial D$, there are two
open balls $B_{1}, \, B_{2}$ with radii $R$ such that $B_{1} \subset D$, $B_{2}
\subset \mathbb{R}^{m}-\bar{D}$ and $\bar{B}_{1} \cap
\bar{B}_{2}=\{x_{0}\}$.
\end{definition}

\begin{theorem}\label{T:main}
Let $D \subset \mathbb{R}^{m}$, $m \geq 2$, be an open set with
$R$-smooth boundary $\partial D$ and finite Lebesgue measure
$\vert D \vert$. Then for all $t>0$,
\[\bigg\vert H_{D}(t) - \vert D \vert + \pi^{-1/2}\mathcal{H}^{m-1}(\partial D)t^{1/2}\bigg\vert \leq
m^{3}2^{m+2}\vert D \vert R^{-2}t.\]
\end{theorem}
Bounds of this type have been obtained in \cite{mvdB87} for the
trace of the Dirichlet heat semigroup and in \cite{vdBD89} for the
heat content with Dirichlet boundary conditions respectively.

Examples of open sets with infinite volume but with finite heat
content in $\R^m$ for any $t>0$ have been given in Theorem 4 of
\cite{mvdB13}. The mechanism for that phenomenon is very different
from the one where one imposes Dirichlet cooling conditions on the
boundary.  Efficient cooling takes place if the boundary is not
too thin. The latter condition can be phrased in terms of a
capacitary density condition for the boundary or a strong Hardy
condition for the quadratic form associated to the Dirichlet
Laplacian. For further details we refer to
\cite{vdB05,vdB07,vdBG04} and \cite{vdBGGK}. In the setting of
\eqref{e1}, \eqref{e2} and \eqref{e3} above, the mechanism for
efficient heat loss is that the complement is not too small. With
this in mind, we introduce the following.

\begin{definition}\label{def2} For an open set $D \subset \R^m$, $R>0$ and
$x\in D$ we define
\begin{equation*}
\mu_D(x;R)=|B(x;R)\cap D|,
\end{equation*}
where $B(x;R) = \{y \in \R^{m} : \vert x - y \vert < R\}$.
\end{definition}
Our second result is the following.

\begin{theorem}\label{the2} Let $D$ be an open set in $\R^m$, and let
$t>0$. Then

\noindent \textup{(i)} $H_D(t)<\infty$ if and only if
$\int_Ddx \, \mu_D(x;(8mt)^{1/2})<\infty$. If
$\int_Ddx \, \mu_D(x;(8mt)^{1/2})<\infty$ then there exist constants
$c_1(m)$ and $c_2(m)$ such that
\begin{equation}\label{e10d}
c_1(m)t^{-m/2}\int_Ddx \, \mu_D(x;(8mt)^{1/2})\le H_D(t)\le
c_2(m)t^{-m/2}\int_D \, dx \, \mu_D(x;(8mt)^{1/2}),
\end{equation}
where
\begin{equation*}
c_1(m)=e^{-2m}(4\pi)^{-m/2},
\end{equation*}
and
\begin{equation*}
c_2(m)=\left(1-2^me^{-3m/2}\right)^{-1}(4\pi)^{-m/2}.
\end{equation*}
\noindent \textup{(ii)} If $0<t_2\le t_1$ and if $H_D(t_1)<\infty$
then
\begin{equation}\label{e10g}
H_D(t_2)\le
\frac{c_2(m)}{c_1(m)}\left(\frac{t_1}{t_2}\right)^{m/2}H_D(t_1).
\end{equation}
\end{theorem}
Since $t\mapsto H_D(t)$ is monotonically decreasing (see
\cite{mP03}), (ii) implies that if $H_D(t)$ is finite for some
$t>0$ then it is finite for all $t>0$. This is in contrast with
the situation where the boundary is kept at fixed temperature $0$.
See for example Theorem 5.5 in \cite{vdBD89}.

While Theorem \ref{the2} holds in the case where $D$ has finite Lebesgue measure,
we have the trivial upper bound
\begin{equation}\label{e10g1}
H_D(t)\le |D|
\end{equation}
in that case, which is sharper than the upper bound in
\eqref{e10d}. Similarly we have by Proposition 9(i) in \cite{mvdB13} that
\begin{equation}\label{e10g2}
H_D(t)\ge |D|-2^{m/2}\int_Ddx \, e^{-\delta(x)^2/(8t)},
\end{equation}
where
$$\delta(x) = \min \{ \vert x - y \vert : y \in \R^{m} -
D\}.$$
The latter implies $\liminf_{t\downarrow 0} H_D(t)=|D|,$
which is sharper than the lower bound in \eqref{e10d}. So Theorem
\ref{the2} is of interest only in the case where $D$ has infinite
Lebesgue measure.

The choice $R=(8mt)^{1/2}$ in the application of Definition
\ref{def2} in Theorem \ref{the2} is natural since diffusion or
transport of heat for small $t$ takes place on a timescale
$t^{1/2}$. However, the constant $(8m)^{1/2}$ is somewhat
arbitrary: any sufficiently large numerical constant would
suffice. Of course the choice of constant affects the numerical
values of $c_1(m)$ and $c_2(m)$ in Theorem \ref{the2}.

We recall the examples in Theorem 4 of \cite
{mvdB13} with finite heat content in $\R^m$ and with infinite volume.
Let $\Sigma$ be an open, bounded convex set in $\R^{m-1}$ and let
\begin{equation}\label{e10h}
\Omega(\alpha, \Sigma)=\{(x,x')\in \R^m:x>1, x'\in
x^{-\alpha}\Sigma\},
\end{equation}
where $\alpha$ is a fixed positive constant. Then $|\Omega(\alpha,
\Sigma)|=+\infty$ if and only if $\alpha<(m-1)^{-1}$. We have the
following.
\begin{proposition}\label{prop} Let $\Omega(\alpha, \Sigma)$ as in \eqref{e10h}
and let $t>0$. Then
\begin{equation}\label{e10i}
H_{\Omega(\alpha, \Sigma)}(t)<\infty
\Leftrightarrow\alpha>(2(m-1))^{-1}.
\end{equation}
If $(2(m-1))^{-1}<\alpha<(m-1)^{-1}$ then
\begin{equation}\label{e10j}
H_{\Omega(\alpha, \Sigma)}(t)\asymp t^{((m-1)\alpha-1)/(2\alpha)},
 \, t\downarrow 0.
\end{equation}
\end{proposition} The precise asymptotic behaviour of $H_{\Omega(\alpha,
\Sigma)}(t)$ was obtained in Theorem 4 (ii) of \cite{mvdB13}.
However, the range of $\alpha$ for which we shall prove
\eqref{e10i} and \eqref{e10j} was incorrectly stated in parts (i)
and (ii) of that theorem.

If $D$ has finite Lebesgue measure then it is convenient to define the {\it heat loss of $D$ in $\R^m$ at $t$} by
\begin{equation}\label{e10k}
F_D(t)=|D|-H_D(t).
\end{equation}
We see that \eqref{e10g1} and \eqref{e10g2} imply that
\begin{equation*}
0\le F_D(t)\le 2^{m/2}\int_Ddx\ e^{-\delta(x)^2/(8t)}.
\end{equation*}
Moreover, $t\mapsto F_D(t)$ is strictly increasing.
In Theorem \ref{the3} below we identify the quantity which controls the heat loss of $D$ in $\R^m$ at $t$, and which motivates the following.
\begin{definition}\label{def3} For an open set $D \subset \R^m$ with finite Lebesgue measure, $R>0$, and
$x\in D$ we define
\begin{equation*}
\nu_D(x;R)=|B(x;R)\cap (\R^m-D)|.
\end{equation*}
\end{definition}
\begin{theorem}\label{the3} If $D$ is an open set in $\R^m$ with finite Lebesgue measure then there exist constants $d_1(m)$ and $d_2(m)$ such that
\begin{equation}\label{e10n}
d_1(m)t^{-m/2}\int_Ddx \, \nu_D(x;4(mt)^{1/2})\le F_D(t)\le d_2(m)t^{-m/2}\int_Ddx \, \nu_D(x;4(mt)^{1/2}),
\end{equation}
where
\begin{equation*}
d_1(m)=e^{-4m}(4\pi)^{-m/2},
\end{equation*}
and
\begin{equation*}
d_2(m)=\left(1-2^{(m+2)/2}e^{-2m}\right)^{-1}(4\pi)^{-m/2}.
\end{equation*}
\end{theorem}

We conclude this section with the following.
\begin{proposition}\label{L:prop0}
Let $D$ be an open set in $\R^{m}$ with finite Lebesgue measure,
and let $1\le p<\infty$. Then \eqref{e2} implies
$\lim_{t\downarrow 0}\lVert u_D(\cdot;t)-
\mathds{1}_{D}(\cdot)\rVert_{L^p(\R^m)}=0$.
\end{proposition}
\begin{proof} We note that $0<u_D(x;t)\le1$. Hence $|u_D(x;t) - 1|\le 1$, and
\begin{align*}
\lVert u_D(\cdot;t)-
\mathds{1}_{D}(\cdot)\rVert_{L^p(\R^m)}^p&=\int_{D}dx \,
|u_D(x;t)-1|^p+ \int_{\R^{m} - D}dx \, u_D(x;t)^p
\\ & \le \int_Ddx \, |u_D(x;t)-1|+\int_{\R^m-D}dx \, u_D(x;t) \\ &
=\int_Ddx \, |u_D(x;t)-1|+\int_{\R^m}dx \, u_D(x;t)-\int_Ddx \,
u_D(x;t)
\\& =2\int_Ddx \, |u_D(x;t)-1|.
\end{align*}
The assertion follows by Lebesgue's Dominated Convergence
Theorem and \eqref{e2}.
\end{proof}

This paper is organised as follows. In Section \ref{sec2} we state
and prove some geometric facts for open sets with $R$-smooth
boundary and finite Lebesgue measure (Proposition \ref{L:prop1}).
We defer the proof of Theorem~\ref{T:main} to Section \ref{sec3},
where we give lower and upper bounds for $u_{D}(x;t)$
(Lemma~\ref{L:lowersol}, Lemma~\ref{L:uppersol}). We then
integrate these bounds with respect to $x$ to obtain lower and
upper bounds for $H_{D}(t)$ (Lemma~\ref{L:lowerhc},
Lemma~\ref{L:upperhc}) which imply Theorem~\ref{T:main} as
advertised. We defer the proofs of Theorem \ref{the2} and of
Theorem \ref{the3} to Section \ref{sec4} respectively. There we
will also prove Proposition \ref{prop}.

\section{Geometrical results for open sets with $R$-smooth boundary and finite Lebesgue
measure}\label{sec2} In what follows, $\omega_{m}$ denotes the
volume of a unit ball $B(x;1)$ in $\R^{m}$ and we define
\begin{equation*}
\diam(D)=\sup\{|x-y|:x\in D, y\in D\},
\end{equation*}
\begin{equation*}
D_{r}=\{x \in D : \delta(x)>r\}
\end{equation*}
and, for $\epsilon > 0$,
\begin{equation*}
D(\epsilon) = \cap_{\{r \, : \, r<\epsilon\}} D_{r} =\{x\in D :
\delta (x) \geq \epsilon\}.
\end{equation*}

\begin{proposition}\label{L:prop1}
If $D \subset \mathbb{R}^{m}$, $m\geq 2$ is an open set with
$R$-smooth boundary $\partial D$ and finite Lebesgue measure
$\vert D \vert$, then:\\
\noindent \textup{(i)} $D$ has at most $\omega_m^{-1}R^{-m}|D|$
components.\\
\noindent \textup{(ii)} If $D^{i}$ is a component of $D$, then
\begin{equation}\label{e11}
\diam (D^{i}) \leq
\frac{\vert D^{i} \vert + (2\omega_{m-1}-\omega_{m})R^{m}}{\omega_{m-1}R^{m-1}},
\end{equation}
with equality if $D^{i}$ is an $R$-neighbourhood of a straight line segment.\\
\noindent \textup{(iii)} $D$ is a $C^{1,1}$ domain.\\
\noindent \textup{(iv)} $\Hdim^{m-1}(\partial D)=\mathcal{P}(D).$ \\
\noindent \textup{(v)} For $0 < r < R$,
\begin{equation}
\Hdim^{m-1}(\partial D^{i})\left(\frac{R-r}{R}\right)^{m-1} \leq
\Hdim^{m-1}(\partial (D^{i}_{r})) \leq \Hdim^{m-1}(\partial
D^{i})\left(\frac{R}{R-r}\right)^{m-1}, \label{e18}
\end{equation}
and
\begin{equation}\label{e19b}
\Hdim^{m-1}(\partial D)\le \frac{m|D|}{R}.
\end{equation}
\end{proposition}
\begin{proof}
\noindent (i) Let $D=\cup_{\{i \in I\}} D^{i}$ where $i \in I$ and $I$ is some
indexing set. Since the components of $D$ are open, $I$ is at most
countable. Let $n(I)$ denote the number of elements in $I$, i.e.
the number of components of $D$. By the $R$-smoothness of
$\partial D$, we have that for each $x \in \partial D$, there is a
ball $B_{1}$ of radius $R$ such that $x \in \bar{B}_{1}$ and
$B_{1} \subset D^{i}$ for some $i \in I$. From this we deduce that
$\vert D^{i} \vert \geq \vert B_{1} \vert =\omega_{m} R^{m}.$ So
\[\vert D \vert = \sum_{i \in I} \vert D^{i} \vert \geq \sum_{i \in I}\omega_{m} R^{m} = n(I)\omega_{m}R^{m}\]
which implies the required estimate.\\

\noindent (ii) We now consider an arbitrary component and denote it by $D$. Since
$D$ is open and connected in $\mathbb{R}^m$, $D$ is path
connected. Let $x_0\in D(R)$ be arbitrary. If $D$ is not bounded
then there exists a path $\gamma:[0,\infty)\rightarrow D$ with
$\gamma(0)=x_0$ and $\lim_{t\rightarrow \infty}\vert \gamma(t)-x_0
\vert = +\infty$. Let $x_n\in \gamma([0,\infty)),\ n\in \N$ be
such that
\begin{equation*}
\vert x_n - x_0 \vert= 4nR.
\end{equation*}
By the triangle inequality
\begin{equation}\label{e15}
\vert x_n-x_k \vert \geq 4\vert n-k \vert R.
\end{equation}
If $B(x_j;R)$ is not contained in $D$, then let $z_j$ be a point
of $\partial D$ such that $\vert z_j-x_j \vert = \delta(x_j) < R$.
By the $R$-smoothness of $\partial D$, there exists a point $y_j
\in D$ such that $B(y_j;R)\subset D$ and $\vert z_j - y_j \vert
=R$. Then $\vert y_j - x_j \vert <R$. If $B(x_j;R)$ is contained
in $D$ then we put $y_j=x_j$. This defines a sequence of points
$(y_j)$ in $D$ such that $\vert y_j - x_j \vert <R$. By
\eqref{e15} and the triangle inequality we have that, for $n \neq
k$
\begin{equation*}
\vert y_n - y_k\vert \geq \vert \vert x_n - x_k \vert - \vert y_n
- x_n \vert- \vert y_k - x_k\vert \vert \geq \vert 4 \vert n-k \vert
R-2R \vert \geq 2 \vert n-k \vert R.
\end{equation*}
We conclude that open balls $\{B(y_n;R)$, $n\in \mathbb{N}\}$ are pairwise
disjoint and contained in $D$. Hence for any $N\in \mathbb{N}$,
$D\supset\cup_{n=1}^NB(y_n;R)$. So $\vert D \vert \geq N\omega_{m}R^m$, which is
a contradiction for $N$ sufficiently large. Therefore $D$ is bounded.

For any point $z\in\partial D$, let $z_R$ be such that $\vert z -
z_R \vert =R$ and $B(z_R;R)\subset D$. The map $z\mapsto z_R$ is
continuous. Since $\partial D$ is path connected, its image under
$z\mapsto z_R$ is path connected. We conclude that both $\partial
D(R)$ and $D(R)$ are path connected. Let $p$ and $q$ be points in
$\partial D$ such that $\diam (D)=\vert p - q \vert$. Let
$\varphi:[0,\infty)\rightarrow D(R)$ be a path from $p_R$ to
$q_R$. Let $P_{\xi},\ 0 \leq \xi \leq \diam (D)$, be the plane
perpendicular to the line segment $[p,q]$ such that it intersects
this line segment at a point which has distance $\xi$ to $p$. For
$R \leq \xi \leq \diam (D)-R$, $P_{\xi}$ intersects $\varphi$ in a
point which belongs to $D(R)$. Hence for $R \leq \xi \leq
\diam(D)-R$,\ $\mathcal{H}^{m-1}(P_{\xi}\cap
D)\ge\omega_{m-1}R^{m-1}$. We conclude that
\begin{align*}
\vert D \vert &=\int_0^{\diam (D)} d\xi \,
\mathcal{H}^{m-1}(P_{\xi}\cap
D)\\
&\geq \frac{1}{2}\omega_{m}R^{m} + \int_R^{\diam (D)-R} d\xi \,
\mathcal{H}^{m-1}(P_{\xi}\cap
D)+\frac{1}{2}\omega_{m}R^{m}\\
&\geq \omega_{m}R^{m} + \int_R^{\diam (D)-R} d\xi \,\omega_{m-1}R^{m-1}\\
&=\omega_{m}R^{m} + \omega_{m-1}(\diam (D)-2R)R^{m-1}.
\end{align*}
This implies \eqref{e11}, and the assertion of equality for an
$R$-neighbourhood of a straight line segment.\\
\noindent (iii) We have by (i) and (ii) that $D$ is bounded. Since $D$ is $R$-smooth we infer by Lemma 2.2 in \cite{AKSZ} that $D$ is a $C^{1,1}$ domain.\\
\noindent (iv) This follows from the fact that $D$ is a bounded $C^{1,1}$ domain and by Remark (ii) on p.183 in \cite{EG92}.\\
\noindent (v) Inequality \eqref{e18} was proved in Lemma 5 of
\cite{mvdB87} for a bounded, open set in $\R^{m}$ with $R$-smooth
boundary. As a consequence of \eqref{e18}, we have that (see
(6.15) in \cite{vdBD89})
\begin{equation*}
\Hdim^{m-1}(\partial D^{i}) \leq \frac{m \vert D^{i} \vert}{R}.
\end{equation*}
Since the components of $D$ and their respective boundaries are disjoint, we have that
\begin{equation*}
\Hdim^{m-1}(\partial D)=\sum_i \Hdim^{m-1}(\partial
D^i)\le \sum_i\frac{m|D^i|}{R}=\frac{m|D|}{R}.
\end{equation*}
\end{proof}
\section{Proof of Theorem~\ref{T:main}}\label{sec3}
Firstly, we set up a convenient coordinate system for the
calculations which follow. Let $x \in D$ satisfy $\delta(x)<
\frac{R}{2}$, and $x_{0} \in \partial D$ be the point such that
$\delta(x) = \vert x - x_{0} \vert$. Choose coordinates $(\zeta,
\hat{\zeta})$ of $\mathbb{R}^{m}$ where $\zeta$ is the direction
of the outward pointing normal to $\partial D$ at $x_{0}$,
$\hat{\zeta}$ represents the $m-1$ orthogonal directions to
$\zeta$ and $x=(0,0)$. In these coordinates, the centre of $B_{1}$
is $(-(R-\delta(x)),0)$ and the centre of $B_{2}$ is
$(\delta(x)+R,0)$. We suppress the $x$-dependence of $B_1$ and
$B_2$ respectively throughout.
\\[5pt]
\indent In the following two lemmas we give
pointwise lower and upper bounds for $u_{D}(x;t)$ respectively
when $\delta(x)<\frac{R}{2}.$
\begin{lemma}\label{L:lowersol}
Let $D \subset \mathbb{R}^{m}$, $m \geq 2$, be an open set with
$R$-smooth boundary $\partial D$ and let $x \in D$ such that
$\delta(x)<\frac{R}{2}$. Then
\begin{align*}
\int_{D}dy\, p(x,y;t) &\geq 1 -(4\pi t)^{-1/2}
\int_{\delta(x)}^{\infty}d\zeta \, e^{-\zeta^{2}/(4t)}
- \frac{\sqrt{2}}{2}e^{-R^2/(8t)}\\
&\ \ \ - (4\pi t)^{-m/2}\int_{-R} ^{\delta(x)}d\zeta \,
e^{-\zeta^2/(4t)}\int_{\{\vert \hat{\zeta} \vert >
\left((\delta(x)- \zeta)\frac{R}{2}\right)^{1/2}\}}d\hat{\zeta} \,
e^{-\vert \hat{\zeta} \vert^ 2/(4t)}.
\end{align*}
\end{lemma}
\begin{proof}
First suppose that $D$ is connected. The set of points  $(\zeta,
\hat{\zeta})\in \R^m$ with fixed $\zeta \in [\delta(x)-2R,
\delta(x)]$ intersects the ball $B_{1}$ in an $(m-1)$-dimensional
ball of radius $\eta$, where $\eta =
(\delta(x)-\zeta)^{1/2}(2R+\zeta-\delta(x))^{1/2}$. Below we use
the fact that for $\zeta > -R$ and $\delta (x) <\frac{R}{2}$,
$\eta \geq \left((\delta(x)-\zeta)\frac{R}{2}\right)^ {1/2}$. So
\begin{align*}
&\int_{D} dy \, p(x,y;t) \geq \int_{B_{1}} dy \, p(x,y;t)\nonumber
\\ &=(4\pi t)^{-m/2}\int_{\delta(x)-2R}^{\delta(x)} d\zeta \,
e^{-\zeta^{2}/(4t)}
\int_{\{\vert \hat{\zeta}\vert < \eta\}} d\hat{\zeta}e^{-\vert \hat{\zeta}\vert^{2}/(4t)}\\
&\geq (4\pi t)^{-m/2}\int_{-R}^{\delta(x)} d\zeta \,
e^{-\zeta^{2}/(4t)} \int_{\{\vert \hat{\zeta}\vert <
\left((\delta(x)-\zeta)\frac{R}{2}\right)^{1/2}\}}
d\hat{\zeta}e^{-\vert \hat{\zeta}\vert^{2}/(4t)}\\
&=(4\pi t)^{-1/2}\int_{-R}^{\delta(x)} d\zeta \,
e^{-\zeta^{2}/(4t)}\nonumber \\ &\hspace{15mm} - (4\pi
t)^{-m/2}\int_{-R}^{\delta(x)} d\zeta \, e^{-\zeta^{2}/(4t)}
\int_{\{\vert \hat{\zeta}\vert
> \left((\delta(x)-\zeta)\frac{R}{2}\right)^{1/2}\}}
d\hat{\zeta}e^{-\vert \hat{\zeta}\vert^{2}/(4t)}\\
&= 1 -(4\pi t)^{-1/2}\int_{\delta(x)}^{\infty} d\zeta \,
e^{-\zeta^{2}/(4t)}
-(4\pi t)^{-1/2}\int_{-\infty}^{-R} d\zeta \, e^{-\zeta^{2}/(4t)}\\
&\hspace{15mm}-(4\pi t)^{-m/2}\int_{-R}^{\delta(x)} d\zeta \,
e^{-\zeta^{2}/(4t)} \int_{\{\vert \hat{\zeta}\vert >
\left((\delta(x)-\zeta)\frac{R}{2}\right)^{1/2}\}}
d\hat{\zeta}e^{-\vert \hat{\zeta}\vert^{2}/(4t)}\\
&\geq 1 -(4\pi t)^{-1/2}\int_{\delta(x)}^{\infty}d\zeta \,
e^{-\zeta^{2}/(4t)}
- \frac{\sqrt{2}}{2}e^{-R^{2}/(8t)}\\
&\hspace{15mm} - (4\pi t)^{-m/2}\int_{-R} ^{\delta(x)}d\zeta \,
e^{-\zeta^{2}/(4t)}\int_{\{\vert \hat{\zeta} \vert >
\left((\delta(x)-
\zeta)\frac{R}{2}\right)^{1/2}\}}d\hat{\zeta} \, e^{-\vert \hat{\zeta} \vert ^{2}/(4t)}.\\
\end{align*}
If $D$ is not connected, then $D = \cup_{i=1}^{N} D^{i}$ for some
$N \in \mathbb{N}$. Since $x \in D$, $x \in D^{i}$ for some $i \in
\{1, \dots, N\}$ with $\delta(x)<\frac{R}{2}$. By the
$R$-smoothness of the boundary, there is a ball $B_{1}^{i}$ of
radius $R$ such that $B_{1}^{i} \subset D^{i}$. So we recover
Lemma \ref{L:lowersol} in that case by the computation above.
\end{proof}
\begin{lemma}\label{L:uppersol}
Let $D \subset \mathbb{R}^{m}$, $m \geq 2$, be an open set with
$R$-smooth boundary $\partial D$, and let $x \in
D$ such that $\delta(x)<\frac{R}{2}$. Then
\begin{align*}
\int_{D}dy\, p(x,y;t) &\leq 1 -(4\pi
t)^{-1/2}\int_{\delta(x)}^{\infty} d\zeta \,
e^{-\zeta^{2}/(4t)} + \frac{\sqrt{2}}{2}e^{-R^{2}/(8t)}\\
&\ \ \ + (4\pi t)^{-m/2}\int_{\delta(x)}^{\delta(x)+R} d\zeta \,
e^{-\zeta^{2}/(4t)} \int_{\{\vert \hat{\zeta} \vert >
((\zeta-\delta(x))R)^{1/2}\}} d\hat{\zeta} \, e^{-\vert
\hat{\zeta} \vert^{2}/(4t)}.
\end{align*}
\end{lemma}
\begin{proof} First suppose that $D$ is connected. Let $H_{x} = \{(\zeta,\hat{\zeta}) \, \vert \, -\infty < \zeta < \delta(x)\}$,
$\tilde{H}_{x} = \{(\zeta,\hat{\zeta}) \, \vert \, \delta(x) +R <
\zeta < \infty\}$ and $S_{x}$ be the slice of width $R$ with
$\partial S_{x}$ parallel to $\partial H_{x}$ excluding the ball
$B_{2}$. The set of points  $(\zeta, \hat{\zeta})\in \R^m$ with
fixed $\zeta \in [\delta(x), \delta(x)+R]$ intersects the ball
$B_{2}$ in an $(m-1)$-dimensional ball of radius $\eta$, where
$\eta=(\zeta-\delta(x))^{1/2}(2R+\delta(x)-\zeta)^{1/2}$. Below we
use the fact that for $\zeta < \delta(x) +R$, $\eta \geq
((\zeta-\delta(x))R)^{1/2}$. So
\begin{align*}
\int_{D}dy \, p(x,y;t)
&\leq \int_{\mathbb{R}^{m}-B_{2}} dy \, p(x,y;t)\\
&\leq \int_{H_{x}} dy \, p(x,y;t) + \int_{\tilde{H}_{x}} dy \, p(x,y;t) + \int_{S_{x}} dy \, p(x,y;t)\\
&=(4\pi t)^{-1/2}\int_{-\infty}^{\delta(x)} d\zeta \,
e^{-\zeta^{2}/(4t)}
+(4\pi t)^{-1/2}\int_{\delta(x)+R}^{\infty} d\zeta \, e^{-\zeta^{2}/(4t)}\\
&\ \ \ +(4\pi t)^{-m/2}\int_{\delta(x)}^{\delta(x)+R} d\zeta \,
e^{-\zeta^{2}/(4t)}
\int_{\{\vert \hat{\zeta} \vert > \eta\}} d\hat{\zeta} \, e^{-\vert \hat{\zeta} \vert^{2}/(4t)}
\end{align*}
\begin{align*}
&\leq 1 - (4\pi t)^{-1/2}\int_{\delta(x)}^{\infty} d\zeta \,
e^{-\zeta^{2}/(4t)}
+\frac{\sqrt{2}}{2}e^{-R^{2}/(8t)}\\
& \ \ \ + (4\pi t)^{-m/2}\int_{\delta(x)}^{\delta(x)+R} d\zeta \,
e^{-\zeta^{2}/(4t)} \int_{\{\vert \hat{\zeta} \vert >
((\zeta-\delta(x))R)^{1/2}\}} d\hat{\zeta} \,
e^{-\vert \hat{\zeta} \vert^{2}/(4t)}.\\
\end{align*}
If $D$ is not connected, then $D=\cup_{i=1}^{N} D^{i}$ for some $N
\in \mathbb{N}$. Since $x \in D$, $x \in D^{i}$ for some $i \in
\{1, \dots, N\}$ such that $\delta(x)<\frac{R}{2}$. By the
$R$-smoothness of the boundary, there is a ball $B_{2}^{i}$ of
radius $R$ such that $B_{2}^{i} \subset \mathbb{R}^{m} -
\bar{D^{i}}$ and $D \subset \mathbb{R}^{m} - B_{2}^{i}$. So we
recover Lemma \ref{L:uppersol} in that case by the computation
above.
\end{proof}
In the following two lemmas, we give lower and upper bounds for
$H_{D}(t)$ which imply Theorem~\ref{T:main}. The $R$-smoothness of
the boundary $\partial D$ ensures that the components $D^{i}$ of
$D$ are sufficiently far apart so that the heat flow from one
component has a negligible effect on the heat flow of another
component. We also use \eqref{e18}, \eqref{e19b} and the additivity
properties of the Hausdorff and Lebesgue measures.
\begin{lemma}\label{L:lowerhc}
Let $D \subset \mathbb{R}^{m}$, $m \geq 2$, be an open set with $R$-smooth boundary $\partial D$ and finite Lebesgue
measure $\vert D \vert$. Then
\begin{align*}
H_{D}(t) &\geq \vert D \vert - \pi^{-1/2}\Hdim^{m-1}(\partial
D)t^{1/2} - 2^{m/2} \big\vert D_{R/2} \big\vert
e^{-R^{2}/(32t)}\nonumber \\ &\ \ \ -
(m-1)2^{m-2}\Hdim^{m-1}(\partial D)R^{-1} \, t -
\frac{\sqrt{2}}{2}\big\vert D - D_{R/2} \big\vert
e^{-R^{2}/(8t)}\nonumber \\ &\ \ \  -
4(m-1)(1+(m-1)2^{m-2})\Hdim^{m-1}(\partial D)R^{-1} \, t.
\end{align*}
\end{lemma}
\begin{proof}
First suppose that $D$ is connected. If $\delta(x) > \frac{R}{2}$
then by Proposition 9(i) in \cite{mvdB13}, we have that
\[\int_{D} dy \, p(x,y;t) \geq 1 - 2^{m/2}e^{-R^{2}/(32t)},\]
and so
\[\int_{D_{R/2}}dx \int_{D} dy \, p(x,y;t) \geq \big\vert D_{R/2} \big\vert
-2^{m/2}\big\vert D_{R/2} \big\vert e^{-R^{2}/(32t)}.\]
If $\delta(x)<\frac{R}{2}$, integrating the result of
Lemma~\ref{L:lowersol} with respect to $x$ over $D - D_{R/2}$ and
using Lemma 6.7 in \cite{vdBD89}, we obtain that
\begin{align*}
&\int_{D - D_{R/2}} dx \int_{D} dy \, p(x,y;t)\nonumber \\ & \geq
\big\vert D - D_{R/2} \big \vert - (4\pi t)^{-1/2}\int_{D -
D_{R/2}} dx \int_{\delta(x)}^{\infty} d\zeta \,
e^{-\zeta^{2}/(4t)} -\frac{\sqrt{2}}{2} \big\vert D - D_{R/2}
\big\vert e^{-R^{2}/(8t)}\nonumber \\ &\ \ \  - (4\pi
t)^{-m/2}\int_{D - D_{R/2}} dx \int_{-R} ^{\delta(x)}d\zeta \,
e^{-\zeta^{2}/(4t)}\int_{\{\vert \hat{\zeta} \vert >
\left((\delta(x)-
\zeta)\frac{R}{2}\right)^{1/2}\}}d\hat{\zeta} \, e^{-\vert \hat{\zeta} \vert ^{2}/(4t)}\\
& = \big\vert D - D_{R/2} \big \vert - (4\pi
t)^{-1/2}\int_{0}^{R/2}dr \, \Hdim^{m-1}(\partial D_{r})
\int_{r}^{\infty}d\zeta \, e^{-\zeta^{2}/(4t)}\nonumber
\end{align*}
\begin{align*}
&\ \ \ -\frac{\sqrt{2}}{2} \big\vert D - D_{R/2} \big\vert e^{-R^{2}/(8t)}\\
&\ \ \  - (4\pi t)^{-m/2}\int_{0}^{R/2} dr \, \Hdim^{m-1}(\partial
D_{r}) \int_{-R} ^{r}d\zeta \, e^{-\zeta^{2}/(4t)}\int_{\{\vert
\hat{\zeta} \vert > \left((r - \zeta)\frac{R}{2}\right)^{1/2}\}}
d\hat{\zeta} \, e^{-\vert \hat{\zeta} \vert ^{2}/(4t)}\\
&\geq \big\vert D - D_{R/2} \big \vert - (4\pi t)^{-1/2}
\Hdim^{m-1}(\partial D)\int_{0}^{R/2}dr \,
\left(\frac{R}{R-r}\right)^{m-1}\int_{r}^{\infty}d\zeta \,
e^{-\zeta^{2}/(4t)}\nonumber \\ &\ \ \ -\frac{\sqrt{2}}{2}
\big\vert D-D_{R/2} \big\vert e^{-R^{2}/(8t)} - (4\pi
t)^{-m/2}\Hdim^{m-1}(\partial D)\nonumber \\ &\ \ \
\times\int_{0}^{R/2} dr \, \left(\frac{R}{R-r}\right)^{m-1}
\int_{-R} ^{r}d\zeta \, e^{-\zeta^{2}/(4t)}\int_{\{\vert
\hat{\zeta} \vert > \left((r - \zeta)\frac{R}{2}\right)^{1/2}\}}
d\hat{\zeta} \, e^{-\vert \hat{\zeta} \vert ^{2}/(4t)}\\
& \geq \big\vert D - D_{R/2} \big \vert - (4\pi
t)^{-1/2}\Hdim^{m-1}(\partial D)\int_{0}^{R/2}dr \,
\left(1+(m-1)2^{m-1}\frac{r}{R}\right)\int_{r}^{\infty}d\zeta \,
e^{-\zeta^{2}/(4t)}\nonumber \\ &\ \ \ -\frac{\sqrt{2}}{2}
\big\vert D -  D_{R/2} \big\vert e^{-R^{2}/(8t)}-
(4\pi t)^{-m/2}\Hdim^{m-1}(\partial D)\nonumber \\
&\ \ \  \times\int_{0}^{R/2} dr \,
\left(1+(m-1)2^{m-1}\frac{r}{R}\right) \int_{-R} ^{r}d\zeta \,
e^{-\zeta^{2}/(4t)}\int_{\{\vert \hat{\zeta} \vert > \left((r -
\zeta)\frac{R}{2}\right)^{1/2}\}}
d\hat{\zeta} \, e^{-\vert \hat{\zeta} \vert ^{2}/(4t)}\\
&\geq \big\vert D-D_{R/2} \big \vert - (4\pi
t)^{-1/2}\Hdim^{m-1}(\partial D)\int_{0}^{\infty}dr
\int_{r}^{\infty}d\zeta \, e^{-\zeta^{2}/(4t)} \\
&\ \ \  - (4\pi t)^{-1/2}(m-1)2^{m-1}\Hdim^{m-1}(\partial
D)R^{-1}\int_{0}^{\infty}dr \, r \int_{r}^{\infty}d\zeta \,
e^{-\zeta^{2}/(4t)}\nonumber \\ & \ \ \  -\frac{\sqrt{2}}{2}
\big\vert D-D_{R/2} \big\vert e^{-R^{2}/(8t)} - (4\pi
t)^{-m/2}(1+(m-1)2^{m-2})\Hdim^{m-1}(\partial D)\nonumber
\\ & \ \ \  \times \int_{-\infty}^{\infty} dr \int_{0} ^{\infty}d\theta \,
e^{-(r-\theta)^{2}/(4t)}\int_{\{\vert \hat{\zeta} \vert >
\left(R\theta/2\right)^{1/2}\}}
d\hat{\zeta} \, e^{-\vert \hat{\zeta} \vert ^{2}/(4t)}\\
&= \big\vert D-D_{R/2} \big \vert -\pi^{-1/2} \Hdim^{m-1}(\partial
D)t^{1/2} -(m-1)2^{m-2} \Hdim^{m-1}(\partial D)R^{-1}t\nonumber \\
&\ \ \
-\frac{\sqrt{2}}{2} \big\vert D-D_{R/2} \big\vert e^{-R^{2}/(8t)}\\
&\ \ \  - (4\pi t)^{-(m-1)/2}(1+(m-1)2^{m-2})\Hdim^{m-1}(\partial
D) \int_{0}^{\infty}d\theta \int_{\{\vert \hat{\zeta} \vert >
\left(R\theta/2\right)^{1/2}\}}d\hat{\zeta} \, e^{-\vert
\hat{\zeta} \vert ^{2}/(4t)}\\ &= \big\vert D-D_{R/2} \big \vert -
\pi^{-1/2}\Hdim^{m-1}(\partial D)t^{1/2} -(m-1)2^{m-2}
\Hdim^{m-1}(\partial D)R^{-1} t\nonumber
\\ &\ \ \  -\frac{\sqrt{2}}{2} \big\vert D-D_{R/2} \big\vert
e^{-R^{2}/(8t)}
 - 4(m-1)(1+(m-1)2^{m-2})\Hdim^{m-1}(\partial D)R^{-1} \, t.\\
\end{align*}
If $D$ is not connected, then $D=\cup_{i=1}^{N} D^{i}$ for some $N
\in \mathbb{N}$ and the above result holds for each component
$D^{i}$ of $D$. By taking the sum of the above result for each
component $D^{i}$ over $i=1,\dots , N$, we obtain
\begin{align*}
&\int_{D} dx \int_{D} dy \, p(x,y;t) \nonumber \\ &\ge
\sum_{i=1}^{N} \left(\big\vert D^{i}_{R/2}\big\vert
-2^{m/2}\big\vert D^{i}_{R/2} \big\vert
e^{-R^{2}/(32t)} + \int_{D^{i}-D^{i}_{R/2}}dx \int_{B_{1}^{i}} dy \, p(x,y;t)\right)\\
&= \big\vert D_{R/2}\big\vert -2^{m/2}\big\vert D_{R/2} \big\vert
e^{-R^{2}/(32t)} +
\sum_{i=1}^{N} \int_{D^{i}-D^{i}_{R/2}} dx \, \int_{B_{1}^{i}} dy \, p(x,y;t)\\
&\geq \vert D \vert - \pi^{-1/2}\Hdim^{m-1}(\partial D)t^{1/2} -
2^{m/2} \big\vert D_{R/2} \big\vert e^{-R^{2}/(32t)}\nonumber \\
&\ \ \ - (m-1)2^{m-2}\Hdim^{m-1}(\partial D)R^{-1}t -
\frac{\sqrt{2}}{2}\big\vert D - D_{R/2} \big\vert
e^{-R^{2}/(8t)}\nonumber \\ &\ \ \ -
4(m-1)(1+(m-1)2^{m-2})\Hdim^{m-1}(\partial D)R^{-1}t.
\end{align*}
\end{proof}
\begin{lemma}\label{L:upperhc}
Let $D \subset \mathbb{R}^{m}$, $m \geq 2$, be an open set with $R$-smooth boundary $\partial D$ and finite Lebesgue
measure $\vert D \vert$. Then
\begin{align*}
H_{D}(t) &  \leq \vert D \vert - \pi^{-1/2}\Hdim^{m-1}(\partial
D)t^{1/2} + 2\pi^{-1/2}\Hdim^{m-1}(\partial D)t^{1/2}
e^{-R^{2}/(32t)}\nonumber \\ &\ \ \  +
2^{-1}(m-1)\Hdim^{m-1}(\partial D)R^{-1}t  +
\frac{\sqrt{2}}{2}\big\vert D-D_{R/2}\big\vert
e^{-R^{2}/(8t)}\nonumber \\ &\ \ \  +
(m-1)(1+(m-1)2^{m-2})\Hdim^{m-1}(\partial D)R^{-1}t.
\end{align*}
\end{lemma}
\begin{proof} First suppose that $D$ is connected. If $\delta(x) > \frac{R}{2}$ then
\[\int_{D_{R/2}} dx \int_{D} dy \, p(x,y;t) \leq \int_{D_{R/2}} dx \int_{\mathbb{R}^{m}} dy \, p(x,y;t)
= \big\vert D_{R/2} \big\vert.\] If $\delta(x)<\frac{R}{2}$,
integrating the result of Lemma~\ref{L:uppersol} with respect to
$x$ over $D-D_{R/2}$ and using Lemma 6.7, in \cite{vdBD89} we
obtain that
\begin{align*}
&\int_{D-D_{R/2}} dx \int_{D} dy \, p(x,y;t)\\
&\leq \big\vert D-D_{R/2}\big\vert - (4\pi
t)^{-1/2}\int_{D-D_{R/2}} dx \int_{\delta(x)}^{\infty} d\zeta
\,e^{-\zeta^{2}/(4t)} + \frac{\sqrt{2}}{2}\big\vert
D-D_{R/2}\big\vert
e^{-R^{2}/(8t)}\\
&\ \ \  + (4\pi t)^{-m/2}\int_{D-D_{R/2}} dx
\int_{\delta(x)}^{\delta(x)+R} d\zeta \, e^{-\zeta^{2}/(4t)}
\int_{\{\vert \hat{\zeta} \vert > ((\zeta-\delta(x))R)^{1/2}\}} d\hat{\zeta} \, e^{-\vert \hat{\zeta} \vert^{2}/(4t)}\\
&= \big\vert D-D_{R/2}\big\vert - (4\pi t)^{-1/2}\int_{0}^{R/2} dr
\, \Hdim^{m-1}(\partial D_{r}) \int_{r}^{\infty} d\zeta
\,e^{-\zeta^{2}/(4t)}\nonumber \\ &\ \ \  +
\frac{\sqrt{2}}{2}\big\vert D-D_{R/2}\big\vert
e^{-R^{2}/(8t)}\\
&\ \ \  + (4\pi t)^{-m/2}\int_{0}^{R/2} dr \, \Hdim^{m-1}(\partial
D_{r}) \int_{r}^{r+R} d\zeta \, e^{-\zeta^{2}/(4t)} \int_{\{\vert
\hat{\zeta} \vert > ((\zeta-r)R)^{1/2}\}} d\hat{\zeta} \,
e^{-\vert \hat{\zeta} \vert^{2}/(4t)}\\
&\leq \big\vert D-D_{R/2}\big\vert - (4\pi
t)^{-1/2}\Hdim^{m-1}(\partial D)\int_{0}^{R/2} dr \,
\left(\frac{R-r}{R}\right)^{m-1} \int_{r}^{\infty} d\zeta
\,e^{-\zeta^{2}/(4t)}\nonumber \\ &\ \ \ +
\frac{\sqrt{2}}{2}\big\vert D-D_{R/2}\big\vert e^{-R^{2}/(8t)}+
(4\pi t)^{-m/2}\Hdim^{m-1}(\partial D)\int_{0}^{R/2} dr \,
\left(\frac{R}{R-r}\right)^{m-1}\nonumber \\ &\ \ \
\times\int_{r}^{r+R} d\zeta \,e^{-\zeta^{2}/(4t)} \int_{\{\vert
\hat{\zeta} \vert > ((\zeta-r)R)^{1/2}\}} d\hat{\zeta} \,
e^{-\vert \hat{\zeta} \vert^{2}/(4t)}\\
&\leq \big\vert D-D_{R/2}\big\vert - (4\pi
t)^{-1/2}\Hdim^{m-1}(\partial D)\int_{0}^{R/2} dr \,
\left(1-(m-1)\frac{r}{R}\right) \int_{r}^{\infty} d\zeta \,
e^{-\zeta^{2}/(4t)}\nonumber \\ &\ \ \  +
\frac{\sqrt{2}}{2}\big\vert D-D_{R/2}\big\vert e^{-R^{2}/(8t)}+
(4\pi t)^{-m/2}\Hdim^{m-1}(\partial D)\nonumber \\
&\ \ \ \times\int_{0}^{R/2} dr
\,\left(1+(m-1)2^{m-1}\frac{r}{R}\right) \int_{r}^{r+R} d\zeta
\,e^{-\zeta^{2}/(4t)} \int_{\{\vert \hat{\zeta} \vert >
((\zeta-r)R)^{1/2}\}} d\hat{\zeta} \,
e^{-\vert \hat{\zeta} \vert^{2}/(4t)}
\end{align*}
\begin{align*}
&\leq \big\vert D-D_{R/2}\big\vert - (4\pi
t)^{-1/2}\Hdim^{m-1}(\partial D)\int_{0}^{\infty} dr
\int_{r}^{\infty} d\zeta \,e^{-\zeta^{2}/(4t)}\nonumber \\ &\ \ \
+ (4\pi t)^{-1/2}\Hdim^{m-1}(\partial D)\int_{R/2}^{\infty}
dr \int_{r}^{\infty} d\zeta \,e^{-\zeta^{2}/(4t)}\\
&\ \ \ +(4\pi t)^{-1/2}(m-1)\Hdim^{m-1}(\partial
D)R^{-1}\int_{0}^{\infty} dr \, r \int_{r}^{\infty} d\zeta \,
e^{-\zeta^{2}/(4t)}\nonumber \\ &\ \ \ +
\frac{\sqrt{2}}{2}\big\vert D-D_{R/2}\big\vert e^{-R^{2}/(8t)} +
(4\pi t)^{-m/2}\Hdim^{m-1}(\partial D)(1+(m-1)2^{m-2})\nonumber
\\ &\ \ \ \times\int_{0}^{\infty} dr \int_{0}^{R} d\theta
\,e^{-(r+\theta)^{2}/(4t)} \int_{\{\vert \hat{\zeta} \vert
> (\theta R)^{1/2}\}} d\hat{\zeta} \,
e^{-\vert \hat{\zeta} \vert^{2}/(4t)}\\
&\leq \big\vert D-D_{R/2}\big\vert -
\pi^{-1/2}\Hdim^{m-1}(\partial D)t^{1/2} +
2\pi^{-1/2}\Hdim^{m-1}(\partial D)t^{1/2}
e^{-R^{2}/(32t)}\nonumber \\ &\ \ \
+2^{-1}(m-1)\Hdim^{m-1}(\partial D)R^{-1}t+
\frac{\sqrt{2}}{2}\big\vert D-D_{R/2}\big\vert
e^{-R^{2}/(8t)}\nonumber \\ &\ \ \  + 2^{-1}(1+(m-1)2^{m-2})(4\pi
t)^{-(m-1)/2}\Hdim^{m-1}(\partial D)\nonumber \\ &\ \ \  \times \int_{0}^{\infty} d\theta
\int_{\{\vert \hat{\zeta} \vert >
(\theta R)^{1/2}\}} d\hat{\zeta} \,e^{-\vert \hat{\zeta} \vert^{2}/(4t)}\\
&= \big\vert D-D_{R/2}\big\vert - \pi^{-1/2}\Hdim^{m-1}(\partial
D)t^{1/2} + 2\pi^{-1/2}\Hdim^{m-1}(\partial
D)t^{1/2}e^{-R^{2}/(32t)}\nonumber \\ & \ \ \
+2^{-1}(m-1)\Hdim^{m-1}(\partial D)R^{-1}t +
\frac{\sqrt{2}}{2}\big\vert D-D_{R/2}\big\vert
e^{-R^{2}/(8t)}\nonumber \\ &\ \ \ +
(m-1)(1+(m-1)2^{m-2})\Hdim^{m-1}(\partial D)R^{-1}t.
\end{align*}
If $D$ is not connected, then $D=\cup_{i=1}^{N} D^{i}$ for some $N \in \mathbb{N}$ and the above result holds for each
component $D^{i}$ of $D$. By taking the sum of the above result for each component $D^{i}$ over $i=1,\dots ,N$, we obtain;
\begin{align*}
&\int_{D} dx \int_{D} dy \, p(x,y;t) \leq \sum_{i=1}^{N}
\left(\big\vert D^{i}_{R/2}\big\vert + \int_{D^{i}-D^{i}_{R/2}}dx
\int_{\mathbb{R}^{m} - B_{2}^{i}}
dy \, p(x,y;t)\right)\\
&= \big\vert D_{R/2} \big\vert + \sum_{i=1}^{N}
\int_{D^{i}-D^{i}_{R/2}}dx \int_{\mathbb{R}^{m} - B_{2}^{i}}
dy \, p(x,y;t)\\
&\leq \vert D \vert - \pi^{-1/2}\Hdim^{m-1}(\partial D)t^{1/2} +
2\pi^{-1/2}\Hdim^{m-1}(\partial D)t^{1/2}
e^{-R^{2}/(32t)}\nonumber \\ &\ \ \  +
2^{-1}(m-1)\Hdim^{m-1}(\partial D)R^{-1}t  +
\frac{\sqrt{2}}{2}\big\vert D-D_{R/2}\big\vert
e^{-R^{2}/(8t)}\nonumber \\ &\ \ \  +
(m-1)(1+(m-1)2^{m-2})\Hdim^{m-1}(\partial D)R^{-1}t.
\end{align*}
\end{proof}
Now using Lemma~\ref{L:lowerhc}, Lemma~\ref{L:upperhc} and
\eqref{e19b}, we obtain that
\[\bigg\vert H_{D}(t) - \vert D \vert +\pi^{-1/2} \Hdim^{m-1}(\partial D)t^{1/2}\bigg\vert \leq
m^{3}2^{m+2}\vert D \vert R^{-2}t,\] which completes the proof of
Theorem~\ref{T:main}.

\section{Proofs of Theorem \ref{the2}, Proposition \ref{prop} and Theorem \ref{the3}.}\label{sec4}

To prove part (i) of Theorem \ref{the2} we suppose the
integrability condition in the hypothesis of Theorem \ref{the2}
holds. We let $R>0$ and $\alpha\in(0,1)$. For $n
\in \N$, we let
\begin{equation*}
D_n=D\cap B(0;n).
\end{equation*}
Then
\begin{align*}
&H_{D_n}(t)=\int_{D_n}dx \int_{D_n} dy \, p(x,y;t)\nonumber
\\& = \int_{D_n}dx\int_{D_n\cap B(x;R)}dy \,
p(x,y;t)+\int_{D_n}dx\int_{D_n\cap (\R^{m} - B(x;R))}dy \,
p(x,y;t)\nonumber \\ & \le (4\pi t)^{-m/2}\int_{D_n}dx \,
\mu_{D_n}(x;R)\nonumber
\\ &\ \ \  +(1-\alpha)^{-m/2}\int_{D_n}dx \, e^{-\alpha
R^2/(4t)}\int_{D_n}dy \, p(x,y;(1-\alpha)^{-1}t)\nonumber \\ & =(4\pi
t)^{-m/2}\int_{D_n}dx \, \mu_{D_n}(x;R)+(1-\alpha)^{-m/2}e^{-\alpha
R^2/(4t)}H_{D_n}((1-\alpha)^{-1}t)\nonumber \\ & \le(4\pi
t)^{-m/2}\int_{D_n}dx \, \mu_{D_n}(x;R)+(1-\alpha)^{-m/2}e^{-\alpha
R^2/(4t)}H_{D_n}(t),
\end{align*}
where we have used that $t\mapsto H_{D_n}(t)$ is decreasing. We
now choose $\alpha$ and $R$ such that $(1-\alpha)^{-m/2}e^{-\alpha
R^2/(4t)}<1$. This is clearly satisfied for $R=(8mt)^{1/2}$ and
$\alpha=\frac{3}{4}$. Since $H_{D_n}(t)\le|B(0;n)|<\infty$, we may
rearrange the terms and obtain that
\begin{align}\label{e43}
H_{D_n}(t)&\le(4\pi t)^{-m/2}\left(1-2^me^{-3m/2}
\right)^{-1}\int_{D_n}dx \, \mu_{D_n}(x;(8mt)^{1/2})\nonumber \\ &
=c_2(m)t^{-m/2}\int_{D_n}dx \, \mu_{D_n}(x;(8mt)^{1/2})\nonumber
\\ & \le c_2(m)t^{-m/2}\int_{D}dx \, \mu_{D}(x;(8mt)^{1/2}).
\end{align}
For $(x,y)\in \R^{2m}$, we let
\begin{equation*}
f_n(x,y)=p(x,y;t)\mathds{1}_{D_n}(x)\mathds{1}_{D_n}(y).
\end{equation*}
Then $(f_n)$ is a monotone increasing sequence of non-negative
functions, converging pointwise to
$p(x,y;t)~\mathds{1}_{D}(x)\mathds{1}_{D}(y)$. The Monotone
Convergence Theorem applied to $(f_n)$ with product measure $dxdy$
gives that $\lim_{n\rightarrow \infty} H_{D_n}(t)=H_D(t)$. This
together with \eqref{e43} implies the right-hand side of
\eqref{e10d}.

To prove the lower bound in Theorem \ref{the2} we have for $R>0$,
\begin{align*}
H_D(t)&\ge \int_Ddx\int_{D\cap B(x;R)}dy \, p(x,y;t)\\ &
\ge (4\pi t)^{-m/2}e^{-R^2/(4t)}\int_Ddx\int_{D\cap
B(x;R)}dy\\& =(4\pi
t)^{-m/2}e^{-R^2/(4t)}\int_Ddx \, \mu_D(x;R).
\end{align*}
The choice $R=(8mt)^{1/2}$ gives the lower bound in \eqref{e10d}.

To prove part (ii) we let $\alpha=\frac{t_2}{t_1}\le 1$, and
suppose $H_D(t_1)<\infty$. Then following \eqref{e43}, we have
that
\begin{align*}
H_{D_n}(\alpha t_1)& \le c_2(m)(\alpha
t_1)^{-m/2}\int_{D}dx \, \mu_{D}(x;(8m\alpha t_1)^{1/2})\\ &
\le c_2(m)t_2^{-m/2}\int_{D}dx \, \mu_{D}(x;(8m t_1)^{1/2}).
\end{align*}
Letting $n\rightarrow \infty$ we obtain that
\begin{equation}\label{e46}
H_{D}(t_2) \le c_2(m)t_2^{-m/2}\int_{D}dx \, \mu_{D}(x;(8m
t_1)^{1/2}).
\end{equation}
By the first inequality in \eqref{e10d} we also have that
\begin{equation}\label{e47}
\int_{D}dx \, \mu_{D}(x;(8m t_1)^{1/2})\le
c_1(m)^{-1}t_1^{m/2}H_D(t_1),
\end{equation}
and \eqref{e10g} follows from \eqref{e46} and \eqref{e47}.
\\[5pt]
\noindent {\it Proof of Proposition \ref{prop}.}
We bound
$\int_{\Omega(\alpha,\Sigma)}dxdx'\mu_{\Omega(\alpha,\Sigma)}((x,x');(8mt)^{1/2})$
from below, and follow the notation of \eqref{e10h}. We restrict
the integral such that $\diam(x^{-\alpha} \Sigma)\le (4mt)^{1/2}$.
That is
\begin{equation}\label{e48}
x\ge \left(\frac{\diam
(\Sigma)}{(4mt)^{1/2}}\right)^{1/\alpha}:=x_0(t).
\end{equation}
We now choose $t$ such that the above set of $x$ satisfies $x\ge
1$. That is
\begin{equation*}
t \le \frac{1}{4m}\diam(\Sigma)^2.
\end{equation*}
We will choose $c>0$ such that for any $x' \in x^{-\alpha}\Sigma$
the cylinder
\begin{equation*}
(x,x+c)\times(x+c)^{-\alpha}\Sigma \subset
B((x,x');(8mt)^{1/2})\cap \Omega(\alpha,\Sigma).
\end{equation*}
Suppose $(z,z')\in (x,x+c)\times(x+c)^{-\alpha}\Sigma$. Then
$z'\in x^{-\alpha}\Sigma$ and
$|(z,z')-(x,x')|^2=(|z-x|^2+|z'-x'|^2)\le c^2+
(\diam(x^{-\alpha}\Sigma))^2\le c^2+4mt$. We conclude that
$(z,z')\in B((x,x');(8mt)^{1/2})$ for $c=(4mt)^{1/2}$.  Since
$\Omega(\alpha,\Sigma)$ is horn-shaped,
$(x,x+c)\times(x+c)^{-\alpha}\Sigma \subset
\Omega(\alpha,\Sigma)$. We have that for all $x$ satisfying
\eqref{e48}
\begin{align*}
\mu_{\Omega(\alpha,\Sigma)}((x,x');(8mt)^{1/2})&\ge
(4mt)^{1/2}\mathcal{H}^{m-1}((x+(4mt)^{1/2})^{-\alpha}\Sigma)\\
&=
(4mt)^{1/2}(x+(4mt)^{1/2})^{-(m-1)\alpha}\mathcal{H}^{m-1}(\Sigma).
\end{align*}
We have that
\begin{align*}
&\int_{\Omega(\alpha,\Sigma)}dxdx' \,
\mu_{\Omega(\alpha,\Sigma)}((x,x');(8mt)^{1/2})
\\ &
\ge \int_{(x_0(t),\infty)}dx\int_{x^{-\alpha}\Sigma}
dx' \, (4mt)^{1/2}(x+(4mt)^{1/2})^{-(m-1)\alpha}\mathcal{H}^{m-1}(\Sigma)
\\ &
=\int_{(x_0(t),\infty)}dx \, (4mt)^{1/2}x^{-(m-1)\alpha}(x+(4mt)^{1/2})^{-(m-1)\alpha}\mathcal{H}^{m-1}(\Sigma)^2.
\end{align*}
The integral diverges for $0<\alpha\le (2(m-1))^{-1}$. For
$(2(m-1))^{-1}<\alpha<(m-1)^{-1}$ we see that the right-hand side
above is of order $x_0(t)^{1-2(m-1)\alpha}$. Using \eqref{e48} and
the lower bound in Theorem \ref{the2} we conclude that there
exists $C_1(\alpha,m,\Sigma)>0$ such that for all $t$ sufficiently
small
\begin{equation}\label{e521}
H_{\Omega(\alpha,\Sigma)}(t)\ge
C_1(\alpha,m,\Sigma)t^{((m-1)\alpha-1)/(2\alpha)}.
\end{equation}

Next we obtain an upper bound for
$\int_{\Omega(\alpha,\Sigma)}dxdx' \,
\mu_{\Omega(\alpha,\Sigma)}((x,x');(8mt)^{1/2})$. We note that for
all $x'\in x^{-\alpha}\Sigma$ and $x\ge 1+(8mt)^{1/2},$
\begin{equation*}
(x-(8mt)^{1/2},x+(8mt)^{1/2})\times(x-(8mt)^{1/2})^{-\alpha}\Sigma
\supset B((x,x');(8mt)^{1/2})\cap \Omega(\alpha,\Sigma).
\end{equation*}
Hence
\begin{align*}
\mu_{\Omega(\alpha,\Sigma)}((x,x');(8mt)^{1/2})&\le
2(8mt)^{1/2}\mathcal{H}^{m-1}((x-(8mt)^{1/2})^{-\alpha}\Sigma)\\
&=
2(8mt)^{1/2}(x-(8mt)^{1/2})^{-(m-1)\alpha}\mathcal{H}^{m-1}(\Sigma).
\end{align*}
We let $x_0(t)$ be as in \eqref{e48},
$(2(m-1))^{-1}<\alpha<(m-1)^{-1}$, and let $t>0$ be such that
$x_0(t)\ge 1+(8mt)^{1/2}$. Then
\begin{align}\label{e55}
&\int_{\{(x,x')\in \Omega(\alpha,\Sigma):x>x_0(t)\}}dxdx' \,
\mu_{\Omega(\alpha,\Sigma)}((x,x');(8mt)^{1/2})\nonumber
\\ &
\le2(8mt)^{1/2}\int_{x_0(t)}^{\infty}dx
\int_{x^{-\alpha}\Sigma}dx' \,
(x-(8mt)^{1/2})^{-(m-1)\alpha}\mathcal{H}^{m-1}(\Sigma)\nonumber
\\
&=2(8mt)^{1/2}\int_{x_0(t)}^{\infty}dx \, x^{-(m-1)\alpha}(x-(8mt)^{1/2})^{-(m-1)\alpha}\mathcal{H}^{m-1}(\Sigma)^2\nonumber
\\
&\le2(8mt)^{1/2}(2(m-1)\alpha-1)^{-1}(x_0(t)-(8mt)^{1/2})^{1-2(m-1)\alpha}\mathcal{H}^{m-1}(\Sigma)^2\nonumber
\\ & \le C_2(\alpha,m,\Sigma)t^{((2m-1)\alpha-1)/(2\alpha)}.
\end{align}
for some finite constant $C_2(\alpha,m,\Sigma)$ and all $t$
sufficiently small.

For all $(x,x')\in \Omega(\alpha,\Sigma)$ with $x\le x_0(t)$ we
bound
\begin{equation*}
\mu_{\Omega(\alpha,\Sigma)}((x,x');(8mt)^{1/2})\le
\omega_m(8mt)^{m/2}.
\end{equation*}
Hence
\begin{align}\label{e57}
&\int_{\{(x,x')\in \Omega(\alpha,\Sigma):1<x<x_0(t)\}}dxdx' \,
\mu_{\Omega(\alpha,\Sigma)}((x,x');(8mt)^{1/2})\nonumber
\\ &
\le\omega_m(8mt)^{m/2}|\{\Omega(\alpha,\Sigma):1<x<x_0(t)\}|\nonumber
\\ &
=\omega_m(8mt)^{m/2}(1-(m-1)\alpha)^{-1}\left(x_0(t)^{1-(m-1)\alpha}-1\right)\mathcal{H}^{m-1}(\Sigma)\nonumber
\\
&\le\omega_m(8
mt)^{m/2}(1-(m-1)\alpha)^{-1}x_0(t)^{1-(m-1)\alpha}\mathcal{H}^{m-1}(\Sigma)\nonumber
\\ &=C_3(\alpha,m,\Sigma) t^{((2m-1)\alpha-1)/(2\alpha)}.
\end{align}

By Theorem \ref{the2}, \eqref{e55} and \eqref{e57} we conclude
that
\begin{equation}\label{e58}
H_{\Omega(\alpha,\Sigma)}(t)\le
c_2(m)\left(C_2(\alpha,m,\Sigma)+C_3(\alpha,m,\Sigma)\right)t^{((m-1)\alpha-1)/(2\alpha)}.
\end{equation}
Proposition \ref{prop} now follows by \eqref{e521} and
\eqref{e58}. \hspace{62mm}$\Box$

The following is a crucial ingredient in the proof of Theorem \ref{the3}.
\begin{proposition}
Let $D$ be an open set in $\R^m$ with finite Lebesgue measure, and
let $s\ge0,\ t\ge0$. Then
\begin{equation}\label{e59}
F_D(s+t)\le F_D(s)+F_D(t).
\end{equation}
\end{proposition}
\begin{proof} By the definition of $F_D(t)$ in \eqref{e10k}, and
\eqref{e6} we recover Preunkert's formula (see \cite{mP03}):
\begin{align}\label{e60}
F_D(t)&= \int_{\R^m}dx \int_D dy\ p(x,y;t)-\int_Ddx \int_Ddy\
p(x,y;t)\nonumber \\ &=\int_{\R^m-D}dx\int_Ddy\ p(x,y;t).
\end{align}
By the heat semigroup property we have that
\begin{equation}\label{e61}
p(x,y;s+t)=\int_{\R^m}dz \, p(x,z;s)p(z,y;t).
\end{equation}
By \eqref{e60}, \eqref{e61} and Tonelli's Theorem we have that
\begin{align*}
F_D(s+t)&=\int_{\R^m-D}dx\int_Ddy\ p(x,y;s+t) \\ &=
\int_{\R^m}dz \,\int_{\R^m-D}dx\int_Ddy\ p(x,z;s)p(z,y;t)
\\ &=\int_{D}dz \,\int_{\R^m-D}dx\int_Ddy\
p(x,z;s)p(z,y;t) \\ &\ \ \ +\int_{\R^m-D}dz
\,\int_{\R^m-D}dx\int_Ddy\ p(x,z;s)p(z,y;t) \\
&=\int_Ddz\int_{\R^m-D}dx\ p(x,z;s)u_D(z;t) \\ &\ \ \ +
\int_{\R^m-D}dz\int_Ddy\ p(z,y;t)u_{\R^{m} - D}(z;s) \\
&\le\int_Ddz\int_{\R^m-D}dx\ p(x,z;s)+\int_{\R^m-D}dz\int_Ddy\
p(z,y;t) \\ &=F_D(s)+F_D(t).
\end{align*}
\end{proof}
We note that in the proof above we did not use the Euclidean
structure. A similar statement and proof would hold for open sets
$D$ in complete Riemannian manifolds which are stochastically
complete. See Section 3.3 in  \cite{Gr} for details on stochastic
completeness.
\\[5pt]
\emph{Proof of Theorem \ref{the3}.} To prove the lower bound we
have by \eqref{e60} that for any $R>0$,
\begin{align*}
F_D(t)&\ge \int_Ddx\int_{(\R^m-D)\cap B(x:R)}dy\ p(x,y;t)\nonumber
\\ & \geq(4\pi t)^{-m/2}e^{-R^2/(4t)}\int_Ddx\ \nu_D(x;R).
\end{align*}
The choice $R=4(mt)^{1/2}$ yields the lower bound in \eqref{e10n}.

To prove the upper bound we have the following estimate.
\begin{align*}
F_D(t)&=\int_Ddx\int_{(\R^m-D)\cap B(x;R)}dy\ p(x,y;t)\nonumber
\\ &\ \ \  +\int_Ddx\int_{(\R^m-D)\cap (\R^m-B(x;R))}dy\ p(x,y;t)\nonumber \\
&\le (4\pi t)^{-m/2}\int_Ddx\ \nu_D(x;R)\nonumber
\\ &\ \ \ +2^{m/2}e^{-R^2/(8t)}\int_Ddx\int_{(\R^m-D)\cap (\R^m-B(x;R))}dy\
p(x,y;2t)\nonumber \\ & \le (4\pi t)^{-m/2}\int_Ddx\
\nu_D(x;R)+2^{m/2}e^{-R^2/(8t)}F_D(2t).
\end{align*}
From \eqref{e59} we infer that $F_D(2t)\le 2F_D(t)$. We conclude
that
\begin{equation}\label{e65}
F_D(t)\le (4\pi t)^{-m/2}\int_Ddx\
\nu_D(x;R)+2^{(2+m)/2}e^{-R^2/(8t)}F_D(t).
\end{equation}
We now choose $R=4(mt)^{1/2}$ so that
$2^{(2+m)/2}e^{-R^2/(8t)}<1$. Rearranging terms in \eqref{e65}
yields the upper bound in \eqref{e10n}. \hspace{79mm}$\Box$

\end{document}